\documentclass[reqno]{amsart}
\usepackage{amssymb,amsmath,amsthm}
\usepackage{color}
\usepackage{enumerate}

\usepackage{fullpage}
\usepackage{url}
\usepackage[Symbol]{upgreek}
\usepackage[unicode,pdfusetitle]{hyperref}
\usepackage[mathscr]{euscript}
\usepackage{mathtools}

\usepackage{dsfont}
\usepackage{relsize}

\usepackage{thmtools}

\usepackage{cleveref}

\usepackage{tikz}
\usetikzlibrary{decorations.pathmorphing, patterns,shapes}
\usepackage{tikz-cd}

\newtheorem{theorem}{Theorem}[section]
\newtheorem*{theorem*}{Theorem}
\newtheorem{question}[theorem]{Question}
\newtheorem*{question*}{Question}

\newtheorem*{conjecture*}{Conjecture}

\newtheorem*{convention*}{Convention}

\newtheorem*{assumption*}{Assumption}
\newtheorem{corollary}[theorem]{Corollary}

\newtheorem*{corollary*}{Corollary}
\newtheorem*{remark*}{Remark}
\newtheorem{proposition}[theorem]{Proposition}
\newtheorem*{proposition*}{Proposition}
\newtheorem{lemma}[theorem]{Lemma}
\newtheorem*{lemma*}{Lemma}
\newtheorem{fact}[theorem]{Fact}
\newtheorem*{fact*}{Fact}

\theoremstyle{definition}
\newtheorem{definition}[theorem]{Definition}
\newtheorem*{definition*}{Definition}
\newtheorem{example}[theorem]{Example}
\newtheorem*{example*}{Example}
\newtheorem{remark}[theorem]{Remark}
\newtheorem{construction}[theorem]{Construction}

\numberwithin{theorem}{section}
\numberwithin{equation}{section}

\DeclareMathOperator{\id}{id}

\DeclareMathOperator{\rc}{rc}

\DeclareMathOperator{\supp}{supp}

\DeclareMathOperator{\Aut}{Aut}
\DeclareMathOperator{\Hom}{Hom}

\DeclareMathOperator{\Int}{Int}

\newcommand{\N}{\mathbb{N}}

\newcommand{\Q}{\mathbb{Q}}
\newcommand{\R}{\mathbb{R}}

\newcommand{\Z}{\mathbb{Z}}

\newcommand{\No}{\mathbf{No}}
\newcommand{\On}{\mathbf{On}}

\newcommand{\uAut}{1\text{-}\!\Aut}
\newcommand{\vAut}{v\text{-}\!\Aut}

\newcommand{\pow}[1]{\!\left(\!\left(#1\right)\!\right)}
\newcommand{\brackets}[1]{\left(#1\right)}

\renewcommand{\geq}{\geqslant}
\renewcommand{\leq}{\leqslant}

\renewcommand{\epsilon}{\varepsilon}
\renewcommand{\phi}{\varphi}

\renewcommand{\rc}{\operatorname{rc}}

\author{Elliot Kaplan}

\author{Lothar Sebastian Krapp}

\author{Michele Serra}

\title{Decomposing the automorphism group of the surreal numbers}
\date{\today}


\begin{document}

\begin{abstract}
We study the automorphism group of the field of surreal numbers.
Our main structure theorem presents a decomposition of this group into a product of five significant factors.
Using the representation of surreal numbers as generalized power series via their Conway normal form, we apply results on Hahn fields and groups from the literature in order to obtain this decomposition.
Moreover, we provide explicit descriptions of the individual factors enabling us to construct automorphisms on the field of surreal numbers from simpler components.
We then extend our study to strongly linear automorphisms in connection to derivations, as well as automorphisms that preserve further exponential structure on the surreals.
\\

On \'etudie le groupe d'automorphismes du corps des nombres surr\'eels.
Notre th\'eor\`eme de structure principal pr\'esente une d\'ecomposition de ce groupe en un produit de cinq facteurs.
En utilisant la repr\'esentation des nombres surr\'eels en s\'eries de puissances g\'en\'eralis\'ees par leur forme normale de Conway, on applique des r\'esultats de la litt\'erature sur les corps et les groupes de Hahn afin d'obtenir cette d\'ecomposition.
De plus, on fournit des descriptions explicites des facteurs individuels, ce qui nous permet de construire des automorphismes du corps des nombres surr\'eels \`a partir de composants plus simples.
On \'etend ensuite notre \'etude aux automorphismes fortement lin\'eaires en lien avec les d\'erivations, ainsi qu'aux automorphismes qui pr\'eservent aussi une structure exponentielle sur les nombres surr\'eels.
\end{abstract}

\maketitle
\section{Introduction}

The ordered field of surreal numbers $\No$ was first introduced and systematically studied in Conway's seminal work~\cite{Co76}.
Initially constructed from basic order principles, the proper class $\No$ can be endowed with a remarkably rich order-theoretic and algebraic structure and serves as a universal domain within the examination of ordered algebraic structures.
Prominently, $\No$ encompasses both the class $\On$ of ordinal numbers as an ordered subclass and the field $\R$ of real numbers as an ordered subfield.
See Gonshor~\cite{Go86} for an extensive treatment of $\No$.

While there are several equivalent representations of surreal numbers, e.g.\ by sign sequences of ordinal length or by cuts between sets of surreals, in this work we consider surreal numbers as generalized power series.
This representation is possible due to the omega map $x\mapsto \omega^x$ on $\No$, which assigns to each surreal the unique simplest positive representative of some archimedean equivalence class.
By means of the omega map, any surreal $a$ can be expressed uniquely in its Conway normal form
$$a=\sum_{\beta<\alpha} r_\beta\cdot \omega^{y_\beta},$$
where $\alpha$ is an ordinal number, $(y_\beta)_{\beta<\alpha}$ is a strictly decreasing sequence in $\No$ and $r_\beta$ are non-zero coefficients from $\R$ (see~\cite[Theorem~21]{Co76},~\cite[Theorem~5.6]{Go86},~\cite[Theorem 14]{Eh01}).
Via the Conway normal form, the field of surreal numbers can be regarded as a field of generalized power series---a concept originally dating back to Hahn's influential work~\cite{Ha07} on non-archimedean ordered systems.
Gonshor himself points out that \emph{``for many purposes we can simply work with these generalized power series and ignore what surreal numbers are in the first place''}~\cite[page~71]{Go86}.
The largest term $r_0\cdot \omega^{y_0}$ in the Conway normal form of $a\neq 0$ determines both the archimedean equivalence class and the sign of $a$:\ the element $\omega^{y_0}$ is the simplest positive representative of the archimedean equivalence class of $a$, and $a$ has positive sign if and only if $r_0$ is positive.
This means that for the natural valuation, the ordered field $\No$ is isomorphic to the ordered Hahn field $\R\pow{\No}$ with ordered residue field $\R$ and ordered additive value group $\No$.

In this work, we initiate a study of the automorphism group of the field $\No$.
Especially for a highly complex field such as the surreals, we hope that the systematic analysis of automorphisms sheds further light on the (algebraic) structure and symmetries of $\No$.
The more recent works~\cite{ks22} and~\cite{ks23} (based on the third author's doctoral thesis~\cite{Se21}) prompted us to start our investigation.
In these works, the automorphism groups of Hahn fields and Hahn groups are systematically dissected.
Based on this structural analysis of automorphism groups, we derive our main result (\Cref{thm:main}) presenting a decomposition of the automorphism group of the field $\No$ into the product of the following five factors:\footnote{The group operation on the vector spaces $\Hom((\No,+),(\R,+))$ and $\R^{\No}$ is given by pointwise addition, and by composition on all other groups.}
\begin{align}\Aut(\No,+,\cdot) \simeq [\uAut(\No,+,\cdot) \rtimes \Hom((\No,+),(\R,+))] \rtimes \left[\Int(\No,+) \rtimes \left[\Aut(\No,<) \times \R^{\No}\right]\right]\label{eq:main}\end{align}
We refer the reader to \Cref{sec:prelims:aut} for a concise introduction of the notation regarding automorphism groups of Hahn fields and Hahn groups.

We establish the decomposition (\ref{eq:main}) in two steps.
First, regarding $\No$ as the Hahn field $(\R\pow{\No},+,\cdot)$, a decomposition into three factors is presented in \Cref{prop:firstdecomp} and simplified afterwards as follows:
$$\Aut(\No,+,\cdot) \simeq [\uAut(\No,+,\cdot) \rtimes \Hom((\No,+),(\R,+))] \rtimes \Aut(\No,+,<)$$
Now the last factor $\Aut(\No,+,<)$ consists of all automorphisms on $\No$ preserving its ordered additive group structure.
Consequently, we regard $\No$ as the ordered additive Hahn group $(\R\pow{\No},+,<)$, for which we obtain in \Cref{prop:groupdec} a decomposition into three factors as follows:
$$ \Aut(\No,+,<) \simeq \Int(\No,+) \rtimes \left[\Aut(\No,<) \times \R^{\No}\right]$$
In all steps of the decomposition, we simplify some of the factors by pointing out isomorphic representations.
For instance by \Cref{rem:homrno}, $(\Hom((\No,+),(\R,+)),+)$ is isomorphic to $(\R^{\No},+)$.
Moreover, we scrutinize all factors further in order to obtain explicit descriptions of automorphisms on $\No$ arising from these particular factors.
Overall, we point out how automorphisms on $\No$ are naturally generated as the composition of simpler automorphisms of $\No$ with specific features with respect to the presentation of $\No$ as a Hahn field or a Hahn group.

In \Cref{sec:str-additive}, we focus on automorphisms in the first factor of decomposition (\ref{eq:main}), in particular, strongly linear ones (automorphisms that commute with infinite sums---see~\cite[Section 3.4]{Se21} for a concise treatment).
We discuss the connection between strongly linear automorphisms and certain derivations on $\No$, based on the correspondence presented in~\cite{BKKPS24}.
This correspondence enables us to construct strongly linear automorphisms (\Cref{ex:derivations}).
We also construct an automorphism that is not strongly linear (\Cref{ex:notstrong}), and we use this to show that strongly linear automorphisms are not dense in the group of all automorphisms with respect to the topology of pointwise convergence (\Cref{rmk:nondense}).
In \Cref{sec:omega}, we consider automorphisms that preserve further order and exponential structure on the surreals.
We show that bijections preserving the total ordering and simplicity ordering are trivial (\Cref{prop:lesslessstrivial}), that field automorphisms preserving the omega-map and valuation are trivial (\Cref{prop:omega-pres-trivial}), and that strongly linear 1-automorphisms preserving the exponential map are trivial (\Cref{prop:trivialexpstrong}).
We leave as questions a complete description of the automorphisms of these richer structures.

\section{Preliminaries}\label{sec:prelims}

\subsection{General conventions}
We denote by $\N$ the set of natural numbers without $0$.
We work in NBG set theory with global choice---a conservative extension of ZFC in which all proper classes are in bijective correspondence with the class $\On$ of ordinals.
By a \textbf{class}, we mean either a set or a proper class; algebraic objects (groups, fields, etc.) may be classes.
For more on working with class-sized models, see~\cite{Eh89}.
For the reader who would prefer not to think about class-sized models, see \Cref{rem:sets}.

\begin{remark}\label{rem:Autdoesn'texist}
In this work, we will consider automorphism groups of proper class-sized models.
We note that any automorphism of such a model is itself a proper class, and so the group of all automorphisms is \textbf{not} a class in NBG set theory (though it is a definable collection of classes).

All statements about decomposing the automorphism group of the surreals (or any other proper class) are to be read as \emph{meta-statements}, which can be translated into first-order statements in the language of NBG about automorphisms themselves (not the group of automorphisms).
For example, the decomposition $\Aut(\No,+,\cdot)\simeq \Int(\No,+,\cdot) \rtimes \mathrm{Ext}(\No,+,\cdot)$ can be translated as ``every automorphism of $(\No,+,\cdot)$ can be expressed uniquely as a composition of an internal and an external automorphism of $(\No,+,\cdot)$.''

Alternatively, one could work in a stronger foundational framework allowing collections of proper classes as objects (for instance, set theories with hyperclasses), in which the automorphism group of $\No$ exists as a genuine object.
We prefer, however, to remain within NBG and to interpret our results accordingly.
The reader who prefers to avoid these issues can also work with $\No(\kappa)$ in place of $\No$, as discussed in \Cref{rem:sets}, since $\Aut(\No(\kappa),+,\cdot)$ is a set in ZFC.
\end{remark}

 All orderings we consider are \textbf{linear orderings}; the one exception is in \Cref{prop:lesslessstrivial}.
Let $(A,<)$ and $(B,<)$ be ordered classes.
We use standard interval notation such as $A^{>a}\coloneqq\{x\in A\mid a<x\}$ for each $a\in A$.
A map $f\colon A\to B$ is \textbf{order preserving} if for any $x,y\in A$ with $x<y$ also $f(x)<f(y)$.
Given two additive abelian groups $(C,+)$ and $(D,+)$, we denote by $\Hom((C,+),(D,+))$ the group of homomorphisms from $C$ to $D$, where the group operation on $\Hom((C,+),(D,+))$ is given by pointwise addition.
Given a field $k$, we let $k^\times \coloneqq k\setminus \{0\}$ denote the group of invertible elements of $k$.

\textbf{Throughout this work, let $k=(k,+,\cdot,<)$ be a set-sized ordered field and let $G=(G,+,<)$ be an ordered abelian group (possibly a proper class).}

\subsection{Hahn fields}\label{sec:valued-fields}

The \textbf{field of generalized power series} (also called \textbf{maximal Hahn field}) with coefficient field $k$ and value group $G$ is denoted by $k\pow{G}$.
It consists of all power series of the form $s=\sum_{g\in G}s_gt^g$, where $s_g\in k$ for each $g\in G$ and the \textbf{support} $\supp(s)=\{g\in G\mid s_g\neq 0\}$ of $s$ is a well-ordered \emph{set}.
We may also take $\alpha$ to be the ordinal representing the order type of $\supp(s)$ and set $s=\sum_{\beta<\alpha}r_\beta t^{g_\beta}$, where $(g_\beta)_{\beta<\alpha}$ is a strictly increasing enumeration of $\supp(s)$ and $r_\beta = s_{g_\beta} \in k^\times$ for all $\beta$.
 For any $s\in k\pow{G}$ and $g\in G$, we denote by $s_g$ the coefficient of $t^g$ in the power series expansion.
We consider $k$ as a subfield of $k\pow{G}$ by identifying each $r\in k$ with $rt^0\in k\pow{G}$.
We endow $k\pow{G}$ with the \textbf{$t$-adic valuation} $v$ mapping any non-zero $s$ to $\min\supp(s)\in G$.
The ordering on $k\pow{G}$ is given by $[s>0 :\Leftrightarrow s_{v(s)}>0]$.
See~\cite[Section~2]{ks22} for further details on Hahn fields.

\begin{remark}
Suppose that $G$ is a proper class (the situation for much of this paper).
As a consequence of our requirement that the support of any $s \in k\pow{G}$ is a set, the field $k\pow{G}$ has more in common with the \emph{$\kappa$-bounded Hahn fields}~\cite{ks05} than with the maximal Hahn fields defined \emph{\`a la} Hahn.
For instance, $k\pow{G}$ may not contain pseudolimits to proper class-length pseudocauchy sequences, and it may support a totally defined exponential.
Elsewhere in the literature, the maximal Hahn field is written $k\pow{G}_{\On}$, to indicate the boundedness condition on the supports, but we omit this subscript.
\end{remark}

\begin{remark}\label{rem:sets}
We note that our results apply just as well to the field $\No(\kappa)$, the \emph{set} of surreal numbers of length $<\kappa$, for any uncountable regular cardinal $\kappa$.
This is a consequence of~\cite[Proposition 4.7]{DE01}, which establishes that $\No(\kappa) = \R\pow{\No(\kappa)}_{\kappa}$ is a $\kappa$-bounded Hahn field that is canonically isomorphic to its own value group.
\end{remark}

The \textbf{natural valuation} on $G$ is denoted by $v_G$.
It maps each element $g\in G$ to its \textbf{archimedean equivalence class} $\{h\in G\mid |g|<n|h|<n^2|g|\text{ for some }n\in \N\}$.
Its value set $\{v_G(g)\mid g\in G\setminus\{0\}\}$ is ordered by $[v_G(g)<v_G(h) :\Leftrightarrow (|g|>|h| \wedge v_G(g)\neq v_G(h))]$.
If $k$ is archimedean, then the natural valuation on $(k\pow{G},+,<)$ coincides with the $t$-adic valuation.
In this case, $v$ is the finest convex valuation on $k\pow{G}$.

\subsection{Automorphisms}\label{sec:prelims:aut}

The automorphism group (where the group operation is given by composition of functions) of an algebraic structure $\mathcal{A}$ is denoted by $\Aut(\mathcal{A})$.
We always indicate the structure that is preserved by the automorphisms explicitly.
More precisely, 
\begin{itemize}
\item $\Aut(k\pow{G},+,\cdot,<)$ consists of all order preserving field automorphisms on $k\pow{G}$, 
\item $\Aut(k\pow{G},+,\cdot)$ consists of all field automorphisms on $k\pow{G}$, 
\item $\Aut(G,+,<)$ consists of all order preserving group automorphisms on $G$, 
\item $\Aut(G,+)$ consists of all group automorphisms on $G$,
\item $\Aut(G,<)$ consists of all order preserving bijections on $G$.
\end{itemize}

{In the rest of this subsection, we define some groups of automorphisms that will be relevant in the sequel.
We refer to~\cite[Section~3]{ks22} and~\cite[Chapter~2]{Se21} for details.}

\begin{definition}
The group $\vAut(k\pow{G},+,\cdot)$ of all \textbf{valuation preserving (field) automorphisms} consists of all automorphisms $\sigma\in \Aut(k\pow{G},+,\cdot)$ satisfying for any $a,b\in k\pow{G}$, $$v(a)<v(b)\Leftrightarrow v(\sigma(a)) < v(\sigma(b)).$$
Given $\sigma\in \vAut(k\pow{G},+,\cdot)$, the \textbf{induced automorphisms} $\sigma_k\in \Aut(k,+,\cdot)$ and $\sigma_G\in \Aut(G,+,<)$ are given by (see~\cite[Remark 2.1.7]{ks22})
\begin{align*}
\sigma_G\colon& G\to G, g\mapsto v(\sigma(t^g)),\\
\sigma_k\colon& k\to k, a\mapsto [\sigma(at^0)]_0.
\end{align*}
\end{definition}

\begin{definition}
The group $\mathrm{Ext}(k\pow{G},+,\cdot)$ of \textbf{external (field) automorphisms} consists of all automorphisms on $(k\pow{G},+,\cdot)$ of the form
$$\sum_{g\in G}s_gt^g\mapsto \sum_{g\in G}\rho(s_g)t^{\tau(g)}$$
for some $\rho\in \Aut(k,+,\cdot)$ and $\tau\in \Aut(G,+,<)$.
The group $\Int(k\pow{G},+,\cdot)$ of \textbf{internal (field) automorphisms} consists of all $\sigma\in \vAut(k\pow{G},+,\cdot)$ inducing the identity maps, i.e.\ $\sigma_k=\id_k$ and $\sigma_G=\id_G$.
Its subgroup $\uAut(k\pow{G},+,\cdot)$ of \textbf{$1$-automorphisms} consists of all $\sigma\in \Int(k\pow{G},+,\cdot)$ with $[\sigma(a)]_{v(a)}=a_{v(a)}$ for any $a\in k\pow{G}^\times$.
\end{definition}

\begin{remark}\label{rmk:intaut}
An automorphism $\sigma\in \vAut(k\pow{G},+,\cdot)$ is internal if and only if for any $s\in k\pow{G}$ the following holds:\ $v(\sigma(s))=v(s)$, and if $v(s)=0$, then $[\sigma(s)]_0=s_0$.
Thus, an automorphism $\sigma\in \vAut(k\pow{G},+,\cdot)$ is a $1$-automorphism if and only if for any $s\in k\pow{G}^\times$ we have $s_{v(s)}t^{v(s)}=\sigma(s)_{v(\sigma(s))}t^{v(\sigma(s))}$.
Alternatively, $\sigma$ is a $1$-automorphism if and only if $\sigma$ fixes the leading term of each $s \in k\pow{G}^\times$.
\end{remark}

Next we turn to the ordered additive group reduct $(k\pow{G},+,<)$ of $(k\pow{G},+,\cdot,<)$.
While the theory we briefly introduce below applies to the context of more general \textbf{Hahn groups}, for our purposes it suffices to focus on Hahn groups that are reducts of Hahn fields.
We refer the reader to~\cite[Chapter~2]{Se21} for the general case.

\begin{definition}
The group of all \textbf{valuation preserving (group) automorphisms} $\vAut(k\pow{G},+)$ consists of all automorphisms $\tau\in \Aut(k\pow{G},+)$ satisfying for any $g,h\in k\pow{G}$, 
\[
v(g)<v(h) \Leftrightarrow v(\tau(g)) < v(\tau(h)).
\]

Given $\tau\in \vAut(k\pow{G},+)$, the induced automorphisms $\tau_{G}\in \Aut(G,<)$ and $\tau_\gamma\in \Aut(k,+)$ for each $\gamma\in G$ are given by
\begin{align*}
\tau_{G}\colon& G\to G, \gamma\mapsto v(\tau(t^{\gamma})),\\
\tau_\gamma\colon& k\to k, a\mapsto [\tau(at^{\gamma})]_{\tau_{G}(\gamma)}.
\end{align*}
\end{definition}

\begin{definition}
The group $\mathrm{Ext}(k\pow{G},+)$ of \textbf{external (group) automorphisms} consists of all automorphisms on $(k\pow{G},+)$ of the form
$$\sum_{\gamma\in G}s_\gamma t^\gamma \mapsto \sum_{\gamma\in G}\eta_\gamma(s_\gamma)t^{\zeta(\gamma)}$$
for some $\zeta\in \Aut(G,<)$ and some $\eta_\gamma\in\Aut(k,+)$ for each $\gamma\in G$.
The subgroup $\mathrm{Ext}(k\pow{G},+,<)$ of $\mathrm{Ext}(k\pow{G},+)$ consists of all order preserving external (group) automorphisms.
The group $\Int(k\pow{G},+)$ of \textbf{internal (group) automorphisms} consists of all $\tau\in \vAut(k\pow{G},+)$ with $\tau_G=\id_G$ and $\tau_\gamma=\id_k$ for each $\gamma\in G$.
\end{definition}

\begin{remark}\label{rmk:intgroup}
An automorphism $\tau\in \vAut(k\pow{G},+)$ is internal if and only if for any $s\in k\pow{G}^\times$ we have $s_{v(s)}t^{v(s)}=[\tau(s)]_{v(\tau(s))}t^{v(\tau(s))}$.
Thus,
$\uAut(k\pow{G},+,\cdot)\leq \Int(k\pow{G},+)$, and indeed,
\[
\Int(k\pow{G},+) \cap \Int(k\pow{G},+,\cdot) = \uAut(k\pow{G},+,\cdot),
\]
where the above groups are viewed as subgroups of the ambient group $\Aut(k\pow{G},+,\cdot)$.
Of course, there are many automorphisms in $\Int(k\pow{G},+)\setminus \uAut(k\pow{G},+,\cdot)$, such as multiplication by $1+\epsilon$ for any nonzero $\epsilon \in k\pow{G}$ with $v(\epsilon)>0$.
Note that any $\tau\in \Int(k\pow{G},+)$ is order preserving, as $[\tau(s)]_{v(\tau(s))}>0$ if and only if $s_{v(s)}>0$.
\end{remark}

\subsection{Surreal numbers}

The theory introduced in \Cref{sec:prelims:aut} applies to $\No$ by setting $(k,+,\cdot,<)=(\R,+,\cdot,<)$ and $(G,+,<)=(\No,+,<)$.
We do so by means of the Conway normal form representation and by identifying $\omega^{-1}$ with the formal variable $t$ in our generalized power series notation.
In other words, we identify $(\No,+,\cdot,<)$ with the ordered Hahn field $(\R\pow{\No},+,\cdot,<)$ via the order preserving field isomorphism
$$\No\to \R\pow{\No}, \sum_{\beta<\alpha} r_\beta\cdot \omega^{y_\beta}\mapsto \sum_{\beta<\alpha} r_\beta t^{-y_\beta}.$$
Using this identification, we obtain
$$(\No,+,\cdot,<)=(\R\pow{\No},+,\cdot,<)\text{ and }(\No,+,<)=\left(\R\pow{\No},+,<\right).$$

\begin{remark}\label{rmk:stronghomog}
We note here that $(\No,+,\cdot,<)$, $(\No,+,<)$, and $(\No,<)$ are all \emph{strongly homogeneous}, meaning that for any set-sized subfield (respectively subgroup, subset) $K\subseteq\No$ and any ordered field (ordered group, ordered set) embedding $\sigma\colon K\to \No$, there is an ordered field (ordered group, ordered set) automorphism of $\No$ extending $\sigma$.
This holds more generally for any o-minimal expansion of $\No$ (see~\cite[Lemma~6.14]{Kr19}) such as the surreal exponential field, which will be considered in \Cref{sec:omega}.
For more on homogeneity when working with proper class-sized models, see~\cite{Eh89}.
\end{remark}

\section{Decomposing the automorphism group}

The first aim of this section is to apply the decomposition of $\vAut(k\pow{G},+,\cdot)$ into four factors as presented in~\cite[Theorem~3.7.1]{ks22} to the field $\No$.
To begin with, we show in \Cref{prop:vautisaut} below that $\vAut(\No,+,\cdot)$ is already the full automorphism group $\Aut(\No,+,\cdot)$ of $\No$.

\begin{lemma}[{\cite[Proposition~1.5.6]{KS22}}]\label{lem:rcimpliesorder}
Let $F$ be a real closed field.
Then any $\sigma\in \Aut(F,+,\cdot)$ is order preserving, i.e.\ $\sigma$ preserves the unique order $<$ on $F$.
Thus, $\Aut(F,+,\cdot)=\Aut(F,+,\cdot,<)$.
\end{lemma}

\begin{lemma}\label{lem:orderimpliesvalue}
Any $\tau\in \Aut(G,+,<)$ preserves $v_G$, i.e.\ for any $a,b\in G$ we have $v_G(a)<v_G(b)$ if and only if $v_G(\tau(a))<v_G(\tau(b))$.
\end{lemma}

\begin{proof}
Let $a,b\in G$.
Then
\begin{align*}
v_G(a)<v_G(b) & \Leftrightarrow \forall n\in \N\colon |a|>n|b|\\
& \Leftrightarrow \forall n\in \N\colon |\tau(a)|>n|\tau(b)|\\
& \Leftrightarrow v_G(\tau(a))<v_G(\tau(b)).\qedhere
\end{align*}
\end{proof}

\begin{proposition}\label{prop:vautisaut}
Suppose that $k\pow{G}$ is real closed and that $k$ is archimedean.
Then 
$\Aut(k\pow{G},+,\cdot)=\vAut(k\pow{G},+,\cdot)$.
In particular,
$$\Aut(\No,+,\cdot)=\vAut(\No,+,\cdot).$$
\end{proposition}

\begin{proof}
Let $\sigma\in \Aut(k\pow{G},+,\cdot)$.
\Cref{lem:rcimpliesorder} implies that $\sigma\in \Aut(k\pow{G},+,<)$.
Thus by \Cref{lem:orderimpliesvalue}, $\sigma$ is valuation preserving.
We obtain $\sigma\in\vAut(k\pow{G},+,\cdot)$.
\end{proof}

In order to apply~\cite[Theorem~3.7.1]{ks22}, we need to check that $\No=\R\pow{\No}$ satisfies the first and canonical second lifting property.
We refer the reader to~\cite[Definitions~3.1.3 \& 3.3.5]{ks22} for the definitions of lifting properties.
Indeed,~\cite[Example~3.1.4~(i)]{ks22} yields that $\R\pow{\No}$ has the first lifting property, and by~\cite[Example~3.3.9~(i)]{ks22} it has the canonical second lifting property.

\begin{proposition}\label{prop:firstdecomp}
\begin{align*}
\Aut(\No,+,\cdot) \simeq [\uAut(\No,+,\cdot) \rtimes \Hom((\No,+),(\R^\times,\cdot))] \rtimes \Aut(\No,+,<).
\end{align*}
\end{proposition}

\begin{proof}
As $\No=\R\pow{\No}$ has the first and canonical second lifting property, we can apply~\cite[Theorem~3.7.1]{ks22} to obtain
\begin{align*}\vAut(\No,+,\cdot)&\simeq \Int(\No,+,\cdot) \rtimes \mathrm{Ext}(\No,+,\cdot)\\ 
\ & \simeq [\uAut(\No,+,\cdot) \rtimes \Hom((\No,+),(\R^\times,\cdot))] \rtimes [\Aut(\R,+,\cdot)\times \Aut(\No,+,<)].\end{align*}
It remains to apply \Cref{prop:vautisaut} and to note that $\Aut(\R,+,\cdot)$ is the trivial group, as any automorphism on the field of real numbers is already the identity.
\end{proof}

The factor $\Hom((\No,+),(\R^{\times},\cdot))$ can be further simplified.

\begin{lemma}\label{lem:expcorresp}
$\Hom((\No,+),(\R^{\times},\cdot))\simeq \Hom((\No,+),(\R,+))$.
\end{lemma}

\begin{proof}
Let $\chi\in \Hom((\No,+),(\R^\times,\cdot))$.
Then for any $a\in \No$ we have
$\chi(a)=\chi(2\cdot a/2)=\chi(a/2)^2\in \R^{>0}$.
We obtain $\Hom((\No,+),(\R^\times,\cdot)) = \Hom((\No,+),(\R^{>0},\cdot))$.
The lemma follows immediately from the fact that the standard exponential map $\exp$ on $\R$ is a group isomorphism from $(\R,+)$ to $(\R^{>0},\cdot)$.
\end{proof}

\begin{corollary}\label{cor:seconddecomp}
$\Aut(\No,+,\cdot) \simeq [\uAut(\No,+,\cdot) \rtimes \Hom((\No,+),(\R,+))] \rtimes \Aut(\No,+,<).$
\end{corollary}

\Cref{prop:firstdecomp} and \Cref{cor:seconddecomp} are the first step of the structural decomposition of $\Aut(\No,+,\cdot)$ into three (non-trivial) factors.
Indeed, examples of non-trivial $1$-automorphisms are given in \Cref{ex:derivations}, we give a full description of $\Hom((\No,+),(\R,+))$ in \Cref{rem:homrno}, and we further decompose $\Aut(\No,+,<)$ in \Cref{prop:groupdec}.
First, let us demonstrate exactly how to construct an element of $\Aut(\No,+,\cdot)$ from members of these three groups.

\begin{construction}\label{constr:2}
Let $\sigma\in \uAut(\No,+,\cdot)$, $\phi\in \Hom((\No,+),(\R,+))$ and $\tau\in \Aut(\No,+,<)$.

Then we obtain an automorphism $\theta(\sigma,\phi,\tau)$ corresponding to the decomposition in \Cref{cor:seconddecomp} as follows.
\[
\theta(\sigma,\phi,\tau)\colon s=\sum_{\beta<\alpha} r_\beta t^{y_\beta}\mapsto \sigma\!\brackets{\sum_{\beta<\alpha} \exp(\phi(\tau(y_\beta))) r_\beta t^{\tau(y_\beta)}}\! = \exp(\phi(\tau(y_0)))r_0 t^{\tau(y_0)}(1+ \epsilon(s)),
\]
for some $\epsilon(s)\in\No$ with $v(\epsilon(s) )>0$.
\end{construction}

\begin{remark}\label{rem:homrno}
Since homomorphisms from $(\No,+)$ to $(\R,+)$ are exactly the $\Q$-linear maps from $\No$ to $\R$, the group $\Hom((\No, +), (\R, +))$ can be identified with $\R^{\mathcal{B}}$, provided a basis $\mathcal{B}$ of $\No$ as a $\mathbb{Q}$-vector space.
Now as $\mathcal{B}$ must be a proper class, it is in bijection with $\No$ by global choice, hence we obtain an isomorphism $(\Hom((\No,+),(\R,+)),+)\simeq(\R^{\No},+)$.
\end{remark}

We now turn to further decomposing the final factor $\Aut(\No,+,<)$ in \Cref{cor:seconddecomp}.
To this end, we now consider $\No$ as the ordered Hahn group $(\R\pow{\No},+,<)$.

\begin{proposition}\label{prop:groupdec}
\begin{align*}
\Aut(\No,+,<) &\simeq \Int(\No,+) \rtimes \mathrm{Ext}(\No,+,<)\\
 & \simeq \Int(\No,+) \rtimes \left[\Aut(\No,<) \times \prod_{\gamma\in \No} \Aut(\R,+,<)\right]\\
 & \simeq \Int(\No,+) \rtimes \left[\Aut(\No,<) \times \R^{\No}\right].
\end{align*}
\end{proposition}

\begin{proof}
The Hahn group $(\R\pow{\No},+)$ has the canonical lifting property (see~\cite[Examples~2.2.3~(i)]{Se21} for details).
Thus,
$$\vAut(\No,+) \simeq \Int(\No,+) \rtimes \mathrm{Ext}(\No,+)$$
(see~\cite[Theorem~2.2.17]{Se21} and~\cite[Theorem~3.15]{ks23}).
Moreover, the isomorphism $\Xi\colon \Int(\No,+) \rtimes \mathrm{Ext}(\No,+) \to \vAut(\No,+)$ is given by $(\tau_1,\tau_2)\mapsto \tau_1\circ\tau_2$.
Now $\Int(\No,+)\leq\Aut(\No,+,<)\leq\vAut(\No,+)$ by \Cref{rmk:intgroup} and 
\Cref{lem:orderimpliesvalue}, respectively.
Hence,
\begin{align*}
\Xi^{-1}(\Aut(\No,+,<))&=\{(\tau_1,\tau_2)\in \Int(\No,+) \rtimes \mathrm{Ext}(\No,+)\mid \tau_1\circ\tau_2 \text{ is order preserving}\}\\
&=\{(\tau_1,\tau_2)\in \Int(\No,+) \rtimes \mathrm{Ext}(\No,+)\mid \tau_2 \text{ is order preserving}\}\\
&=\Int(\No,+) {\rtimes} \mathrm{Ext}(\No,+,<),
\end{align*}
establishing the first isomorphism.

Next, the group of external automorphisms decomposes as $$\mathrm{Ext}(\No,+)\simeq\Aut(\No,<) \times \prod_{\gamma\in \No} \Aut(\R,+)$$
via the isomorphism 
$$\Theta\colon \Aut(\No,<) \times \prod_{\gamma\in \No} \Aut(\R,+)\to \mathrm{Ext}(\No,+), \brackets{\zeta,(\eta_\gamma)_{\gamma\in \No}}\mapsto \left[\sum_{\gamma\in \No}s_\gamma t^\gamma \mapsto \sum_{\gamma\in \No}\eta_\gamma(s_\gamma)t^{\zeta(\gamma)}\right]\!.$$
For the group of order preserving external automorphisms, we obtain
\begin{align*}
&\Theta^{-1}(\mathrm{Ext}(\No,+,<))\\
=&\left\{\left.\brackets{\zeta,(\eta_\gamma)_{\gamma\in \No}}\in \Aut(\No,<) \times \prod_{\gamma\in \No}\Aut(\R,+)\ \right| \forall\gamma\in\No\ \forall r\in\R\colon [\eta_\gamma(r)>0\Leftrightarrow r>0] \right\}\\
=&\Aut(\No,<) \times \prod_{\gamma\in \No} \Aut(\R,+,<),
\end{align*}
yielding the second isomorphism.

Finally, note that the group $\Aut(\R,+,<)$ consists of all maps $\R\to\R, x\mapsto ax$ for some $a\in \R^{>0}$.
This is due to the fact that any map in $\Aut(\R,+,<)$ is a monotone solution of Cauchy's (basic) functional equation (see~\cite[Section~2.1]{Ac66}).
Hence, $\Aut(\R,+,<)\simeq (\R^{>0},\cdot)$.
As $(\R^{>0},\cdot)$ is isomorphic to $(\R,+)$ via the real logarithm map $\log$, we obtain 
$$\prod_{\gamma\in \No} \Aut(\R,+,<)\simeq \prod_{\gamma\in \No} (\R,+) =(\R^{\No},+),$$
as required.
\end{proof}

Recall from \Cref{constr:2} how $\theta(\sigma,\phi,\tau)\in \Aut(\No,+,\cdot)$ arises from some $\sigma\in \uAut(\No,\allowbreak+,\allowbreak\cdot)$, $\phi\in \Hom((\No,+),(\R,+))$ and $\tau\in \Aut(\No,+,<)$.
Based on the decomposition of $\Aut(\No,+,<)$ in \Cref{prop:groupdec}, we refine $\tau$ in the following construction as $\tau(\nu,\zeta,g)$ for some internal group automorphism $\nu\in\Int(\No,+)$, some order preserving bijection $\zeta\in \Aut(\No,<)$ and some map $g\colon \No\to \R$.

\begin{construction}\label{constr:3}
Let $\nu\in\Int(\No,+)$, $\zeta\in \Aut(\No,<)$ and $g\in \R^{\No}$.
We construct $\tau(\nu,\zeta,g)$ following the decomposition presented in \Cref{prop:groupdec}:
\[
\tau(\nu,\zeta,g)\colon \sum_{\beta<\alpha}r_\beta t^{y_\beta} \mapsto \nu\!\brackets{\sum_{\beta<\alpha}\exp(g(y_\beta))r_\beta t^{\zeta(y_\beta)}}\! = \exp(g(y_0))r_0 t^{\zeta(y_0)}(1+\epsilon),
\]
for some $\epsilon\in\No$ with $v(\epsilon )>0$.
\end{construction}

Finally, by combining \Cref{cor:seconddecomp} and \Cref{prop:groupdec}, we obtain our main decomposition result.

\begin{theorem}\label{thm:main}
$$\Aut(\No,+,\cdot) \simeq \left[\uAut(\No,+,\cdot) \rtimes \Hom((\No,+),(\R,+))\right] \rtimes \left[\Int(\No,+) \rtimes \left[\Aut(\No,<) \times \R^{\No}\right]\right].$$
\end{theorem}

Given $\sigma\in \uAut(\No,+,\cdot)$, $\phi\in\Hom((\No,+),(\R,+))$, $\nu\in\Int(\No,+)$, $\zeta\in\Aut(\No,<)$ and $g\in\R^{\No}$, one can combine \Cref{constr:2} and \Cref{constr:3} to obtain an automorphism $\theta(\sigma,\phi,\tau(\nu,\zeta,g))\in \Aut(\No,+,\cdot)$ arising from the decomposition presented in \Cref{thm:main}.

Due to \Cref{rem:homrno}, we also obtain the following decomposition, which depends on the choice of a $\Q$-basis $\mathcal{B}$ of $\No$ and a bijection from $\mathcal{B}$ to $\No$.

\begin{corollary}
$$\Aut(\No,+,\cdot) \simeq \left[\uAut(\No,+,\cdot) \rtimes \R^{\No}\right] \rtimes \left[\Int(\No,+) \rtimes \left[\Aut(\No,<) \times \R^{\No}\right]\right].$$
\end{corollary}

Thus, the factor $(\R^{\No},+)$ appears twice in the decomposition.
We point out that the first of those two factors highly depends on a choice of a $\Q$-basis $\mathcal{B}$ of $\No$ and a bijection $f_{\mathcal{B}}\colon \mathcal{B}\to \No$, whereas the second factor arises naturally without any particular choice.

\section{Strongly additive automorphisms}\label{sec:str-additive}
An automorphism $\sigma \in \vAut(k\pow{G},+,\cdot)$ is said to be \textbf{strongly additive} if for all $\sum_{g \in G} s_g t^g \in k\pow{G}$, the family $(\sigma(s_g t^g))_{g \in G}$ is summable (that is to say, the class $\bigcup_{g \in G}\supp(\sigma(s_g t^g))$ is well-ordered and each $g \in G$ appears in only finitely many of these supports), and we have the equality
\[
\sigma\big( \sum_{g \in G} s_g t^g\big) = \sum_{g \in G} \sigma(s_g t^g).
\]
See~\cite[Section 3.4]{Se21} for more background on strongly additive maps.
Let $\vAut^+(k\pow{G},+,\cdot)$ denote the group of strongly additive valuation preserving automorphisms.
Note that every external field automorphism is strongly additive~\cite[Proposition 4.0.9]{ks22}.

A \textbf{strongly $k$-linear} automorphism of $(k\pow{G},+,\cdot)$ is a strongly additive automorphism that is the identity on $k$.
We denote by ${\vAut^+_k(k\pow{G},+,\allowbreak\cdot)}$ the group of strongly $k$-linear valuation preserving automorphisms.

\begin{remark}\label{ex:compositions}
One source of strongly linear automorphisms comes from the \emph{hyperseries program} of Bagayoko, van der Hoeven, and others~\cite{Ba22,Ba22B,BH21,BH22,BHK21,DHK19}.
Conjecturally, there is a composition law 
\[
(f,g) \mapsto f\circ g \colon \No \times \No^{>\R}\to \No
\]
 satisfying some natural conditions, one of which is that for any $g \in \No^{>\R}$, the map $\sigma_g \colon \No\to \No$ given by $\sigma_g(f) \coloneqq f\circ g$ is a strongly $\R$-linear automorphism (see~\cite[page 235]{Ba22B} for a precise conjecture).
\end{remark}

In order to study strongly linear automorphisms systematically, we first examine where they fit into our decomposition.
Let $\sigma \in \uAut(\No,+,\cdot)$, $\tau \in \Aut(\No,+,<)$, and $\phi \in \Hom((\No,+),(\R,+))$, and let $\theta = \theta(\sigma,\phi,\tau)$ be as in \Cref{constr:2}.
If $\sigma$ is the identity, then $\theta$ is strongly $\R$-linear.
Thus, in studying strongly additive and strongly $\R$-linear automorphisms of $\No$, it is enough to focus on $1$-automorphisms.

In~\cite{BKKPS24}, a correspondence is established between strongly $k$-linear $1$-automorphisms of $k\pow{G}$ and certain strongly $k$-linear derivations on $k\pow{G}$.
A \textbf{strongly $k$-linear derivation} is a strongly $k$-linear map 
\[
\partial \colon k\pow{G}\to k\pow{G}
\]
that satisfies $\partial(ab) = a\partial(b) + b \partial(a)$ for all $a,b \in k\pow{G}$.
A derivation $\partial$ is said to be \textbf{contracting} if $v(\partial a)>v(a)$ for all $a \in k\pow{G}\setminus\{0\}$.
The following correspondence was established in the case that $G$ is a set, but it also holds for $G$ a proper class.

\begin{fact}[{\cite[Theorem 3.13]{BKKPS24}}]\label{fact:deraut}
If $\partial$ is a contracting strongly $k$-linear derivation, denote by $\partial^i$ the $i$-th iterate of $\partial$.
The map $\exp(\partial) \colon k\pow{G}\to k\pow{G}$ given by 
\[
\exp(\partial)(a) \coloneqq \sum_{i=0}^\infty \partial^i(a)/i!
\]
is a strongly $k$-linear $1$-automorphism of $k\pow{G}$.
The map $\partial\mapsto \exp(\partial)$ is a bijection between the class of contracting strongly $k$-linear derivations and the class of strongly $k$-linear $1$-automorphisms.
\end{fact}

In many cases, it is easier to work with contracting strongly linear derivations than $1$-automorphisms, as we will see in \Cref{prop:trivialexpstrong} below.
It is also easy to find non-trivial contracting derivations, using the following variant of a construction from~\cite{Ca18,Ha18}.

\begin{example}\label{ex:derivations}
Let $\phi \in \Hom((\No,+),(\R,+))$ and consider the derivation $\partial_\phi\colon\No\to \No$ given by
\[
\sum_{g \in \No}r_g t^{g} \mapsto \sum_{g \in \No}r_g \phi(g)t^{g-1}.
\]
Then $\partial_\phi$ is a contracting strongly $\R$-linear derivation, so $\exp(\partial_\phi)$ belongs to $\uAut(\No,+,\cdot)$.
\end{example}

 We now present an example of a $1$-automorphism of $(\No,+,\cdot)$ that is not strongly additive.

\begin{example}\label{ex:notstrong}
Let $K \subseteq \No$ be the field of Puiseux series over $\R$ in $\omega^{-1}$, so $K$ consists of all surreal numbers of the form $\sum_{i \in \Z} r_i \omega^{i/n}$ where $n \in \N$ and where $r_i = 0$ for all sufficiently large $i$.
Then $K$ is real closed with value group $\Q$.

Let $a \coloneqq \sum_{n \in \N}\omega^{-1/n}$.
Then $a$ is a pseudolimit of the pseudocauchy sequence $(\sum_{n=1}^m \omega^{-1/n})_{m <\omega}$ in $K$, so $K(a)$ and its real closure $K(a)^{\rc}$ are immediate extensions of $K$ (because $K$ is real closed and $a$ is a pseudolimit).
Let $\epsilon \coloneqq \omega^{-\omega}\in \No$, so $0<\epsilon<K(a)^{>0}$.
Thus, $a$ and $a+\epsilon$ realize the same cut over $K$, so there is a unique isomorphism $\sigma_1\colon K(a)^{\rc}\to K(a+\epsilon)^{\rc}$ that is the identity on $K$ and sends $a$ to $a+\epsilon$.

Using global choice, fix a well-indexed $\Q$-vector space basis $(y_\alpha)_{\alpha\in \On}$ of $\No$ with $y_0 = 1$ and $y_1 = \omega$.
For each $\alpha\in \On$, let $\Gamma_\alpha \coloneqq \bigoplus_{\beta < \alpha} \Q y_\beta$, and let $K_\alpha \coloneqq \R\pow{\Gamma_\alpha}$, identified in the usual way with a subfield of $\No$.
We define an automorphism $\sigma_\alpha$ on $K_\alpha$ for each $\alpha\geq 2$ as follows:
\begin{enumerate}
\item The field $K_2 = K\pow{\Q\oplus \Q\omega}$ is an immediate extension of $K(a,\omega^\omega)^{\rc} = K(a,\epsilon)^{\rc} $.
Using that $\omega^{\omega}$ is greater than both $K(a)^{\rc}$ and $K(a+\epsilon)^{\rc}$, we extend $\sigma_1$ to an automorphism of $K(a,\omega^\omega)^{\rc}$ that fixes $\omega^\omega$.
We further extend this to an automorphism $\sigma_2$ of $K_2$ using the uniqueness of maximal immediate extensions.
\item Let $\alpha\geq 2$ and assume that we have defined $\sigma_\alpha$.
Then $K_{\alpha+1}$ is an immediate extension of $K_\alpha(\omega^{y_{\alpha}})^{\rc}$.
We first extend $\sigma_\alpha$ to an automorphism of $K_\alpha(\omega^{y_{\alpha}})^{\rc}$ that fixes $\omega^{y_{\alpha}}$, and then we further extend this to $K_{\alpha+1}$ by uniqueness of maximal immediate extensions.
\item Let $\lambda>0$ be a limit ordinal, and assume we have defined $\sigma_{\alpha}$ for $\alpha<\lambda$.
Then $K_\lambda$ is an immediate extension of $\bigcup_{\alpha<\lambda} K_\alpha$, so the automorphisms $(\sigma_{\alpha})_{\alpha<\lambda}$ extend to an automorphism $\sigma_\lambda$ of $K_\lambda$, again using uniqueness of maximal immediate extensions.
\end{enumerate}
We have $\No = \bigcup_{\alpha\in\On}K_\alpha$, and we let $\sigma\colon \No\to \No$ be the common extension of the maps $(\sigma_\alpha)_{0<\alpha \in \On}$.
Then $\sigma(\omega^y) = \omega^y$ for all $y \in \No$ and $\sigma(r) = r$ for all $r \in \R$, so $\sigma$ is a 1-automorphism.
Clearly, $\sigma$ is not strongly additive since 
\[
\sigma\big( \sum_{n \in \N}\omega^{-1/n}\big) = \sum_{n \in \N}\omega^{-1/n}+ \omega^{-\omega}\neq \sum_{n \in \N}\sigma(\omega^{-1/n}).
\]
\end{example}

Strongly linear automorphisms have the desirable property that they can be easily expressed by declaring them on monomials and extending by strong linearity.
It is thus natural to ask whether non-strongly linear automorphisms can be approximated by strongly linear ones.

\begin{question}\label{ques:dense}
Can $\Aut(\No,+,\cdot)$ be endowed with a reasonable group topology, such that $\Aut^+(\No,+,\cdot)$ is a dense subgroup?
\end{question}

The next remark rules out a natural topology.

\begin{remark}\label{rmk:nondense}
Consider the \emph{topology of pointwise convergence}---the topology with basic open neighborhoods\footnote{As discussed in \Cref{rem:Autdoesn'texist}, $\Aut(\No,+,\cdot)$ is not itself a class in NBG set theory.
Accordingly, when discussing a topology on this group, it does not make sense to use the terminology ``open sets'' or ``open classes'', so we instead say ``open neighborhoods''.
These are definable collections of automorphisms in NBG, closed under finite intersections and arbitrary definable unions.} of the form
\[
U_{\sigma}(y_1,\ldots,y_n) \coloneqq \{\tau \in \Aut(\No,+,\cdot)\mid \tau(y_i) = \sigma(y_i)\text{ for }i = 1,\ldots,n\}
\]
where $n\in \N$, $\sigma \in \Aut(\No,+,\cdot)$, and $y_1,\ldots,y_n \in \No$.
This is the analogue of the topology used for the Galois group of an infinite Galois extension.

Example~\ref{ex:notstrong} shows that $\Aut^+(\No,+,\cdot)$ is not dense in $\Aut(\No,+,\cdot)$ with respect to this topology.
To see this, let $\sigma$ and $a = \sum_{n \in \N}\omega^{-1/n}$ be as in \Cref{ex:notstrong} and consider the open neighborhood
\[
U_\sigma(a,\omega) = \{\tau \in \Aut(\No,+,\cdot) \mid \tau(a) = a+\omega^{-\omega} \text{ and }\tau(\omega) = \omega\}.
\]
For $\tau$ in this neighborhood, we have $\tau(\omega^{-1/n}) = \tau(\omega)^{-1/n} = \omega^{-1/n}$, so 
\[
\tau\big(\sum_{n \in \N}\omega^{-1/n} \big) \neq \sum_{n \in \N}\omega^{-1/n} = \sum_{n \in \N}\tau(\omega^{-1/n}).
\]
Thus, this neighborhood contains no strongly additive automorphisms.
\end{remark}

\section{Preserving exponentiation and the $\omega$-map}\label{sec:omega}
In addition to its structure as a field of generalized power series, the field of surreal numbers enjoys a great deal of additional structure:
\begin{enumerate}
\item There is a well-founded partial ordering $<_s$ on $\No$, and with this ordering, $(\No,<,<_s)$ is a \emph{full lexicographically ordered binary tree}; see~\cite{Eh01}.
\item There is an exponential $\exp$ on $\No$ that extends the real exponential~\cite{Go86}.
With this exponential, $(\No,+,\cdot,\exp)$ is an elementary extension of the field of real numbers~\cite{DE01}.
\item There is the omega map $x \mapsto \omega^x$, mentioned in the introduction.
With this map, $(\No,+,\cdot,\omega^x)$ is a so-called \emph{omega-field}~\cite{BKMM23}.
\end{enumerate}

In this section, we briefly touch on the automorphism groups of these richer structures (though much remains open in the last two cases).

\begin{proposition}[{\cite[Theorem 5.1]{Eh94}}]\label{prop:lesslessstrivial}
The subgroup $\Aut(\No,<,<_s)$ is trivial.
\end{proposition}
\begin{proof}
As a full lexicographically ordered binary tree, $(\No,<,<_s)$ has the following properties:\ every $x \in \No$ has exactly two immediate $<_s$-successors (one $< x$ and one $>x$), and every $<_s$-chain $(y_\alpha)_{\alpha<\lambda}$ in $\No$ with $\lambda$ a limit ordinal has a unique $<_s$-supremum (when $\lambda = 0$ and this chain is empty, then $0$ is this supremum).
Moreover, any surreal number is either an immediate $<_s$-successor of another number, or a $<_s$-supremum of such a chain.

Let $\sigma \in \Aut(\No,<,<_s)$.
If $x \in \No$ is fixed by $\sigma$, then its two immediate $<_s$-successors must be as well.
If $\sigma$ fixes an $<_s$-chain $(y_\alpha)_{\alpha<\lambda}$ with $\lambda$ a limit, then it fixes the unique $<_s$-supremum of this chain.
As $\sigma$ fixes the empty chain, it fixes zero, and thus by induction it fixes all of $\No$.
\end{proof}

Now we turn to automorphisms of $(\No,+,\cdot,\exp)$.
Recall that $(\No,+,\cdot,\exp)$ is strongly homogeneous:\ for any set-sized elementary exponential subfield $K \subseteq \No$ and any elementary embedding $\sigma\colon K\to \No$, there is an automorphism of $(\No,+,\cdot,\exp)$ extending $\sigma$ (see \Cref{rmk:stronghomog}).
While automorphisms of $(\No,+,\cdot,\exp)$ are therefore plentiful, we have not been able to find an example of a non-trivial $1$-automorphism of $(\No,+,\cdot,\exp)$.
This is partially explained by the following negative result.

\begin{proposition}\label{prop:trivialexpstrong}
The group $\uAut^+_\R(\No,+,\cdot,\exp)$ of strongly $\R$-linear exponential field $1$-automorphisms is trivial.
\end{proposition}
\begin{proof}
Let $\partial \colon \No\to \No$ be a contracting strongly $\R$-linear derivation and suppose that $\partial$ is not identically zero.
In light of \Cref{fact:deraut}, it is enough to show that the automorphism $\sigma \coloneqq \sum_{i=0}^\infty \frac{1}{i!}\partial^i$ does not commute with the exponential map.

We claim that we can find $y \in \No$ with $v(\partial y)<0$.
First, take $a \in \No$ with $\partial a \neq 0$ and take $b \in \No$ with $v(b)<\min\{0,-v(\partial a)\}$.
Then 
\[
\partial(ab) = a \partial b + b\partial a,\qquad \partial(ab^2) = b(2a \partial b + b \partial a).
\]
Either $v(a \partial b + b\partial a)$ or $v(2a \partial b + b \partial a)$ is at most $v(b \partial a )<0$, so taking either $y = ab$ or $y = ab^2$, we get $v(\partial y)<0$.

Now with $y$ as above, let $\epsilon \coloneqq \sigma(y) - y$, so 
$\exp(\sigma y) = \exp(y + \epsilon) = \exp(y) \exp(\epsilon)$.
Since $\partial$ is contracting, we have $v(\epsilon ) = v(\partial y )<0$, so $v\exp (\epsilon)\neq 0$.
Therefore,
\[
v \exp (\sigma y) = v\exp (y) + v\exp(\epsilon) \neq v\exp(y) = v(\sigma(\exp y)),
\]
where the last equality uses that $\sigma$ is a $1$-automorphism.
Thus, $\sigma$ does not commute with the exponential.
\end{proof}

Automorphisms obtained from compositions as in \Cref{ex:compositions} are always strongly $\R$-linear and commute with the exponential.
Thus we have the following.

\begin{corollary}
Automorphisms obtained from compositions are never $1$-automorphisms.
\end{corollary}

\begin{question}
Are there any non-trivial $1$-automorphisms of $(\No,+,\cdot,\exp)$?
\end{question}

Finally, we turn to automorphisms that commute with the $\omega$-map.

\begin{proposition}\label{prop:omega-pres-trivial}
Any automorphism of $(\No,+,\cdot, \omega^x)$ that fixes the valuation is trivial.
In particular, the group $\uAut(\No,+,\cdot, \omega^x)$ is trivial.
\end{proposition}
\begin{proof}
Let $\sigma$ be an automorphism preserving the omega map and let $x \in \No$, so $\sigma(\omega^x) = \omega^{\sigma(x)}$.
Recall that we identify $\omega^x$ with $t^{-x} \in \R\pow{\No}$, so $v(\omega^x) = -x$.
Thus, if $v(\sigma(\omega^x)) = v(\omega^x)$, then $\sigma(x) = x$.
\end{proof}

\begin{question}
What are the non-trivial automorphisms of $(\No,+,\cdot, \omega^x)$?
\end{question}
The search for these automorphisms seems related to the class of \textbf{generalized $\epsilon$-numbers} (that is, the fixed points of the map $x \mapsto \omega^x$).
As an ordered class, the generalized $\epsilon$-numbers are in bijective correspondence with $(\No,<)$; see~\cite[Theorem 9.1]{Go86}.
Note that any $\sigma \in \Aut(\No,+,\cdot, \omega^x)$ induces an automorphism of the ordered class of generalized $\epsilon$-numbers:\ if $\omega^x = x$, then $\omega^{\sigma(x)} = \sigma( \omega^x) = \sigma(x)$, so $\sigma(x)$ is a generalized $\epsilon$-number as well.
More generally, Lemire established the existence of a proper class of surreals of the form
\[
x = a_0 \omega^{a_1\omega^{a_2\omega^{\cdot^{\cdot^\cdot}}}}
\]
for any sequence $(a_n) \in \{-1,1\}^\omega$~\cite{Le96}.
These classes are again in bijective correspondence with $(\No,<)$, and are necessarily fixed by any $\sigma \in \Aut(\No,+,\cdot, \omega^x)$.

\begin{question}
Does any order-isomorphism of the generalized $\epsilon$-numbers extend to an automorphism of $(\No,+,\cdot, \omega^x)$?
\end{question}

\subsection*{Acknowledgments}
This project was initiated within the Ontario/Baden-Württemberg Faculty Mobility Program in August 2023 and resumed at the Banff International Research Station in February 2024.
We gratefully acknowledge the involved institutions for partial funding and their hospitality.
The first author is a postdoctoral researcher of the Fonds de la Recherche Scientifique -- FNRS, and was supported in part by the National Science Foundation under Award No.\ DMS-2103240.
Parts of this research were conducted while the first author was hosted by the Max Planck Institute for Mathematics, and he thanks the MPIM for its support and hospitality.
The second author was partially supported by the Vector Stiftung.

We also thank Salma Kuhlmann for suggesting this project, Patrick Speissegger for his support, Vincent Bagayoko for helpful discussions, Micka{\"e}l Matusinski for pointing out \Cref{ques:dense} and the anonymous referee for helpful remarks and comments.

\subsection*{Authorship Contribution Statement:} All authors contributed equally to this work.

\vfill

\textsc{Elliot Kaplan}, Département de Mathématique, Université de Mons, Belgium\\
\textit{E-mail address}: \texttt{elliot.kaplan@umons.ac.be}

\medskip

\textsc{Lothar Sebastian Krapp}, Institut für Interdisziplinäre Sprachevolutionswissenschaft, Universität Zü\-rich, Switzerland \& Fachbereich Mathematik und Statistik, Universität Konstanz, Germany\\
\textit{E-mail address}: \texttt{sebastian.krapp2@uzh.ch}

\medskip

\textsc{Michele Serra}, Fakultät für Mathematik, Technische Universität Dortmund, Germany\\
\textit{E-mail address}: \texttt{michele.serra@tu-dortmund.de}

\end{document}